\documentclass{birkjour}
\usepackage{graphicx}
\usepackage{amssymb}
\newcommand{\D}{\mathcal D}
\newcommand{\F}{\mathcal F}
\newcommand{\lcm}{\mathrm {lcm}}
\newcommand{\legendre}[2]{{\genfrac{(}{)}{1pt}{}{#1}{#2}}}

\newtheorem{theorem}{Theorem}[section]
\newtheorem{lemma}[theorem]{Lemma}
\newtheorem{conjecture}[theorem]{Conjecture}

\newtheorem{proposition}[theorem]{Proposition}
\newtheorem*{RRSP}{First Rogers--Ramanujan identity (series--product form)}
\newtheorem*{RRC}{First Rogers--Ramanujan identity (combinatorial form)}
\newtheorem*{RR2SP}{Second Rogers--Ramanujan identity (series--product form)}
\newtheorem*{RR2C}{Second Rogers--Ramanujan identity (combinatorial form)}
\newtheorem*{EOD}{Euler's partition theorem}

\theoremstyle{definition}

\theoremstyle{remark}
\newtheorem{remark}[theorem]{Remark}

\numberwithin{equation}{section}

\begin{document}

%-------------------------------------------------------------------------
% editorial commands: to be inserted by the editorial office
%
%\firstpage{1} \volume{228} \Copyrightyear{2004} \DOI{003-0001}
%\seriesextraline{This is the Concrete Title of this Book\br H.E. R and S.T.C. W, Eds.}
%\firstpage{1}

\title[A formula for the partition function that ``counts"]{A formula for the partition function\\ that ``counts"}

%----------Author 1

\author{Yuriy Choliy}
\address{%
Department of Chemistry and Chemical Biology\\
Rutgers University New Brunswick---Busch Campus\\
610 Taylor Road\\
Piscataway, NJ 08854, USA}
\curraddr{}
\email{yurkocholiy@hotmail.com}
%\thanks{}

%----------Author 2
\author{Andrew V. Sills}
\address{Department of Mathematical Sciences\\
Georgia Southern University\\
Statesboro, GA 30458, USA\\
voice 1-912-478-5424\\
fax 1-912-478-0654}
\curraddr{}
\email{ASills@georgiasouthern.edu}
\thanks{The work of the second author is partially supported by National Security Agency Grant P13217-39G3217}



\date{\today}

\begin{abstract}
We derive a combinatorial multisum expression for the number $D(n,k)$ of partitions of $n$ with Durfee square of order $k$.  An immediate corollary is therefore a combinatorial formula for $p(n)$, the number of partitions of $n$.  We then study $D(n,k)$ as a quasipolynomial.  We consider the natural polynomial approximation $\tilde{D}(n,k)$ to the quasipolynomial representation of $D(n,k)$.  
Numerically, the sum $\sum_{1\leq k \leq \sqrt{n}} \tilde{D}(n,k)$ appears to be
extremely close to the initial term of the Hardy--Ramanujan--Rademacher convergent series for $p(n)$.
\end{abstract}

%%% ----------------------------------------------------------------------
\maketitle
%%% ----------------------------------------------------------------------
%\tableofcontents

\section{Introduction}
\subsection{Preliminaries: Definitions and Notation}
A \emph{partition} $\lambda$ of an integer $n$ is a nonincreasing finite sequence of positive integers 
$(\lambda_1, \lambda_2, \lambda_3, \dots, \lambda_{\ell})$ that sum to $n$.  Each $\lambda_j$ is called a \emph{part} of $\lambda$.   The number of parts $\ell=\ell(\lambda)$ is the \emph{length} of $\lambda$. The number of times $m_j = m_j(\lambda)$ the positive integer $j$ appears as a part in $\lambda$ is the \emph{multiplicity} of $j$ in $\lambda$.  Sometimes it is convenient to notate the partition $\lambda$ as $\lambda = \langle 1^{m_1} 2^{m_2} 3^{m_3} \cdots \rangle$, with the convention that the superscript $m_j$ may be omitted if $m_j=1$ for that $j$, and $j^{m_j}$ may be omitted if $m_j=0$ for that $j$.  Thus, e.g., $(6,4,3,3,1,1,1,1) = \langle 1^4 3^2 4\ 6 \rangle$.
Also, the sum of the parts of $\lambda$ is denoted $|\lambda|$ and is called the \emph{weight} of
$\lambda$.  

It will be convenient to define the \emph{union} $\lambda\cup\mu$ 
of two partitions $\lambda$ and $\mu$ as
  \[ \lambda\cup\mu = \langle 1^{m_1(\lambda)+ m_1(\mu)} 2^{m_2(\lambda) + m_2(\mu)} 
  3^{m_3(\lambda) + m_3(\mu)} \cdots \rangle. \]
It is well known that the partition $\lambda$ can be represented graphically by an arrangement of 
$\ell(\lambda)$ rows of left-justified dots, with $\lambda_j$ dots in the $j$th row.  Such a collection of dots is called the \emph{Ferrers graph}~\cite[p. 6]{A76} of $\lambda$.  For example, the Ferrers graph associated with the partition
$(6,4,3,3,1,1,1,1)$ is as follows.
\[ \begin{array} {cccccc} 
{\bullet} & \bullet & \bullet & \cdot & \cdot &\cdot \\
\bullet & \bullet& \bullet & \cdot & \\
\bullet & \bullet  & \bullet \\
\cdot & \cdot & \cdot \\
\cdot \\
\cdot \\
\cdot \\
\cdot 
\end{array}
\]

The largest square starting from the upper left, and contained within the Ferrers graph is called the \emph{Durfee square} of $\lambda$~\cite[p. 28]{A76}, \cite[p. 76]{AE04} (indicated above by the darker dots in the the Ferrers graph).   Let us call the number of nodes on a side (or equivalently the number of nodes along the main diagonal) of a Durfee square the \emph{order} of the Durfee square.  

   As it turns out, the Durfee square is of interest outside the theory of partitions.  For example, the
order of the Durfee square is in fact 
equivalent to the \emph{h-index} introduced by Hirsch~\cite{H05} as a metric that attempts to measure both the productivity and impact of a scholar.  If a scholar has published $\ell$ papers that have been cited in the literature at least once each, 
 and the number of citations of the $j$th paper is 
  $\lambda_j$, then $\lambda=(\lambda_1, \lambda_2, \dots, \lambda_\ell)$ is a partition of $n$ (once the $\lambda_j$ are arranged in nonincreasing order), where $n$ is the total number of citations.  The h-index for this author is therefore the order of the Durfee square of $\lambda$.

Since we will be discussing asymptotics later, let us adopt the convention throughout that $q$ represents
a complex variable with modulus less than $1$ (in some cases $q$ may be considered a
formal variable).
Let $D(n,k)$ denote the number of partitions of $n$ with Durfee square of order $k$. 
It is well known~\cite[p. 28]{A76} that the generating function for $D(n,k)$ is 
\begin{equation} \label{DnkGF}
\sum_{n\geq k^2} D(n,k) q^n = \frac{q^{k^2}}{(1-q)^2 (1-q^2)^2 (1-q^3)^2\cdots (1-q^k)^2}. 
\end{equation}

Associated with each partition $\lambda = (\lambda_1, \lambda_2, \dots, \lambda_\ell)$ of $n$ is the \emph{conjugate} partition $\lambda' = 
(\lambda'_1, \lambda'_2, \dots \lambda'_{\lambda_1} ) =\langle 1^{\lambda_1-\lambda_2} 2^{\lambda_2-\lambda_3}
3^{\lambda_3-\lambda_4}\cdots (\ell-1)^{\lambda_{\ell-1}-\lambda_\ell} \ell^{\lambda_\ell} \rangle$ of
$n$ which may be obtained from $\lambda$ by interchanging the rows and columns of the Ferrers graph of $\lambda$.  

  The \emph{Frobenius symbol} of $\lambda$ with Durfee square of order $k$ is the $2\times k$ matrix
\[ \left[ \begin{array}{ccccc}  
a_1 & a_2 & a_3 & \cdots & a_k \\
b_1 & b_2 & b_3 & \cdots & b_k 
\end{array}
  \right], \]
 where $a_j = \lambda_j - j$ and $b_j = \lambda'_j - j$ for $j=1,2,3,\dots, k$.
 Notice that $a_1 > a_2 > \cdots> a_k$, $b_1>b_2>\cdots >b_k$, and $n= k+\sum_{j=1}^k (a_j + b_j )$.
 
  Thus $D(n,k)$ also counts the number of Frobenius symbols with exactly $k$ columns whose entries sum to $n-k$.

\subsection{The partition function}
  The \emph{partition function}, denoted $p(n)$, is the number of partitions of $n$. 
The first exact formula for $p(n)$, which was derived using the theory of modular forms, was given by Hardy and Ramanujan~\cite{HR18} in 1918.  Two decades later, Rademacher~\cite{HR38} modified the 
derivation of the Hardy--Ramanujan formula and as a result, produced an infinite series that converges to $p(n)$:
\begin{equation} \label{HRR}
p(n) = \frac{1}{\pi\sqrt{2}} \sum_{k\geq 1} \sqrt{k} A_k(n) \frac{d}{dn} \left(
  \frac{\sinh\left( \frac{\pi}{k} \sqrt{\frac 23 (n-\frac{1}{24}) } \right)}{\sqrt{n-\frac{1}{24}}} \right), 
\end{equation}
where the Kloosterman-type sum
\[ A_k(n) := \underset{\gcd(h,k)=1}{\sum_{0\leq h<k}} \omega(h,k) e^{-2\pi i n h/k}, \]
and the $24k$th root of unity
\begin{multline*} \omega(h,k) \\
:= \left\{ 
\begin{array}{ll}
  \legendre{-k}{h} \exp\Big(  \frac{\pi i}{12}\left( {3(2-hk-h)+(k-k^{-1})(2h-h'+h^2h')} \right) \Big)  & \mbox{if $2\nmid h$} \\
  \legendre{-h}{k} \exp\Big( \frac{\pi i}{12} \left( {3(k-1) + (k-k^{-1})(2h-h'+h^2h')} \right) \Big) & \mbox{if $2\nmid k$}
\end{array}
  \right.  ,\end{multline*}
with $\legendre{a}{b}$ denoting the Legendre symbol, and $h'$ denoting any solution to
the congruence $hh'\equiv -1\pmod k$.

  In 2011, Ono and Bruinier~\cite{BO13} announced a new formula for $p(n)$ as a finite sum of algebraic integers that are singular moduli for a certain weak Maass form described using Dedekind's eta function and the quasimodular Eisenstein series $E_2$.  Thus the known formulas for $p(n)$ are not combinatorial in nature and require deep complex function theory for their derivation.  
  
  One goal of this present paper is to present the following combinatorial formula for $p(n)$:
\begin{equation} \label{PofN}
p(n) =   \sum_{k=0}^{\lfloor \sqrt{n} \rfloor} D(n,k),
\end{equation}
where $D(n,k)$ is given by the following $(k-1)$-fold sum of terms, each of which is a positive integer:
\begin{equation} \label{DnkMultisum}
D(n,k) = \sum_{m_k=0}^{U_k} \sum_{m_{k-1}=0}^{U_{k-1}} \cdots \sum_{m_2=0}^{U_2}
 \left( 1+n-k^2 - \sum_{h=2}^k h m_h \right) \prod_{i=2}^k (m_i + 1),
\end{equation}
where 
\begin{equation} \label{Uj}
 U_j := U_j(n,k) = \left\lfloor \frac{ n-k^2 - \sum_{h=j+1}^k h m_h}{j} \right\rfloor,
 \end{equation}
for $j = 2,3,4, \dots, k.$

\subsection{Overview}
After proving Equation~\eqref{DnkMultisum}, we study various aspects of $D(n,k)$.
We 
find leading coefficients of the generating function of $D(n,k)$, examine representation of $D(n,k)$ as the sum of quasipolynomials, and suggest a polynomial approximation for $p(n)$ which gives a value very close to that of the first term of Rademacher's convergent series representation of $p(n)$. 
We present several combinatorial expressions related to our multisum formula for $D(n,k)$,
which may be obtained in a straightforward manner by suitable modification of the main result.
 Finally, we compare our multisum representation of $D(n,k)$ with the quasipolynomial representation.

\section{Multisum representation of $D(n,k)$} 

We begin with an elementary lemma for which we provide two proofs.
\begin{lemma} \label{L1}
\begin{equation} \label{DnkPtn}
  D(n,k) = \sum_{\lambda\in\mathcal{P}_{n-k^2,k}} \prod_{i=1}^k \Big( m_i(\lambda) +1 \Big),
\end{equation}
where $\mathcal{P}_{n,r}$ denotes the set of all partitions of weight $n$ in which no part exceeds $r$.
 \end{lemma}
\begin{proof}[Generating function proof]
The truth of~\eqref{DnkPtn} follows immediately from expanding each factor in the denominator
of the generating function of $D(n,k)$ as a binomial series, and then extracting the coefficient of
$q^n$ in the extremes:
\begin{align*}
&\phantom{=}\sum_{n\geq k^2} D(n,k) q^n\\ &= \frac{q^{k^2}}{(1-q)^2 (1-q^2)^2 \cdots (1-q^k)^2} \\
&= \sum_{m_1, m_2, \dots, m_k\geq 0} (m_1+1)(m_2+1)\cdots (m_k+1) q^{k^2+m_1 + 2m_2 + \cdots
+ k m_k}.
\end{align*} 
\end{proof}
\begin{proof}[Combinatorial proof]
Fix $n$ and $k$.  An arbitrary partition $\lambda$ of $n$ with Durfee square of order $k$ may be
decomposed, via its Ferrers graph, into a triple of partitions $(\delta, \beta, \rho)$ where
$\delta = \langle k^k \rangle$, $\beta = ( \lambda_{k+1}, \lambda_{k+2}, \dots , \lambda_{\ell(\lambda)})$, and $\rho = (\lambda'_{k+1}, \lambda'_{k+2}, \dots, \lambda'_{\ell(\lambda')})$.  Thus $\delta$ is
the Durfee square, $\beta$ consists of the parts of $\lambda$ below the Durfee square, and $\rho$ consists of the parts to the right of
the Durfee square formed vertically, i.e. parts below the Durfee square in the conjugate of $\lambda$.
(For example, if $\lambda=(64331111)$, then $(\delta,\beta,\rho) = (333, 31111, 211)$.)

 The subpartition $\delta$ is the same for all partitions $\lambda$ enumerated by $D(n,k)$.  Clearly, 
$\beta \cup\rho$ is a partition of weight $n-k^2$ in which no part exceeds $k$.  
We visualize the Durfee square of order $k$ fixed in place.
The multiplicity $m_k(\beta\cup\rho)$ of $k$'s can be placed in $m_k(\beta\cup\rho)+1$
different ways into the Ferrers diagram of $\lambda$:
all in $\beta$ and none in $\rho$, all but one in $\beta$ and one in $\rho$,
\dots.  For every possible placement of the $k$'s among $\beta$ and $\rho$, there are, by the same reasoning, $m_{k-1}(\beta\cup\rho)+1$ possible ways to place the $k-1$'s among $\beta$ and $\rho$.
And so on for all integers from $k-2$ down to $1$.
\end{proof}

 \begin{theorem}
   Equation~\eqref{DnkMultisum} is valid.
 \end{theorem} 
\begin{proof}  Suppose $\lambda$ is a partition of weight $n-k^2$ with each part at most $k$. 
Then $\lambda$ may contain at most $U_k$ parts equal to $k$.  Once this number $m_k$ of $k$'s is established, then we can say that $\lambda$ 
may contain at most $U_{k-1}$ parts equal to $k-1$, and so on until we arrive at the observation
that $\lambda$ may contain at most $U_2$ $2$'s.  At this point, all remaining parts must equal $1$.
But $n-k^2 = m_1 + 2m_2 + 3m_3 + \cdots + k m_k$, thus 
$m_1 = n - k^2 - (2m_2 + 3m_3 + \cdots + k m_k)$.
Applying the preceding observations to~\eqref{DnkPtn} yields~\eqref{DnkMultisum}.
 \end{proof}

\begin{remark}
Of course, for particular values of $k$, Equation~\eqref{DnkMultisum} can be reduced to a
closed form product using elementary summation formulas:

\begin{align*}
D(n,1)&=n\\
D(n,2) &= \frac{1}{6} \left(-4 \left\lfloor \frac{n}{2}\right\rfloor +3 n-1\right)
   \left(\left\lfloor \frac{n}{2}\right\rfloor -1\right) \left\lfloor
   \frac{n}{2}\right\rfloor
\end{align*}
However, for $k>2$, the formulas are very long, and they obscure the underlying simplicity demonstrated by
the $(k-1)$-fold multisum.
\end{remark}

\section{Quasipolynomial representation of $D(n,k)$}
\subsection{Some notation}
As in Stanley~\cite[p. 474]{RS11}, we say $f: \mathbb{N}\to\mathbb{C}$ is a \emph{quasipolynomial} of degree $d$ if there exists a positive integer $N$ and polynomials $f_0, f_1, \dots, f_{N-1}$ such that
\[ f(n) = f_i(n)  \mbox{ if $n\equiv i \hskip -3mm \pmod{N}$}, \] where $\mathbb{N}$ denotes the nonnegative integers, at least one of the $f_i$ is of degree $d$ and none of the $f_i$ has degree exceeding $d$.   The integer $N$ is called a \emph{quasiperiod} of $f$.

For convenience, let us introduce the notation
\[ \D_k := \D_k(q) = \sum_{n\geq k^2} D(n,k) q^n, \]
and denote the $n$th cyclotomic polynomial as 
  \[ \Phi_n = \Phi_n(q) = \underset{\gcd(n,k)=1}{\prod_{1\leq k \leq n}} (q -e^{2 i \pi k /n}  ). \]
  
  Of course, $\D_k$ admits a partial fraction decomposition of the form
  \begin{equation} \label{DPF} \D_k = \frac{q^{k^2}}{\prod_{j=1}^k \Phi_j^{2\lfloor k/j \rfloor}}
      = \sum_{j=1}^{k} \sum_{\ell=1}^{2\lfloor k/j  \rfloor}  \sum_{h=0}^{\varphi(j)-1} 
      \frac{c_{h,j,\ell}(k) q^h }{\Phi_j^\ell},
       \end{equation} for some rational numbers $c_{h,j,\ell}(k)$ and where $\varphi$ 
       denotes Euler's totient function.

\subsection{The generating function for $D(n,k)$.} 
\begin{remark}The second author and D. Zeilberger have given explicit expressions for $D(n,k)$ as sums of quasipolynomials for numeric $k$ in the range $1\leq k \leq 40$~\cite{SZ12,SZMaple}.  Here, we would like to say something about $D(n,k)$ for general (symbolic) $k$.
\end{remark}

A representation of $D(n,k)$ for any $k$ as a sum of quasipolynomials of quasiperiods
$1,2,3,\dots,k$ arises naturally out of series expansions of the partial fraction decomposition of $\D_k$.
In fact, the series expansion of the terms with denominator $\Phi_j^\ell$, for
$\ell = 1, 2, \dots, 2\lfloor k/j \rfloor$ give a contribution to $D(n,k)$ in the form of a 
quasipolynomial of quasiperiod $j$ and degree $2\lfloor k/j \rfloor -1$. %when $k\geq j$.

 We begin by deriving some results about this decomposition.

\begin{lemma}\label{L}
  Some coefficients in the partial fraction decomposition of $\D_k$ include:
  \begin{align}
     c_{0,1,2k}(k)  &=\frac{1}{(k!)^2}   \label{c012k}\\
     c_{0,1,2k-1}(k) &= \frac{k+1}{2(k-1)! k!} \label{c012km1}\\
     c_{0,1,2k-2}(k) &= \frac{9k^2+25k+13}{72(k-2)! k!} \label{c012km2} \\
     c_{0,1,2k-3}(k) &= \frac{3k^4+10k^3-4k^2-31k-14}{144 (k-2)! k!}  \label{c012km3}%\\
 %    c_{0,1,2k-4}(k) & = \frac{675k^6 + 2475k^5 -4524k^4 -21839k^3 + 701k^2+41561k+16086}
 %    {259200(k-2)!k!}. \label{c012km4}
  \end{align}
  \end{lemma}
  \begin{proof}
    Since $\D_k$ has a pole of order $2k$ at $q=1$, it will be convenient to define $\F_k = \F_k(q) := (1-q)^{2k} \D_k,$ so that $\F_k$ is analytic at $q=1$ and thus has a Taylor series expansion there:
    \[ \F_k = \sum_{j=0}^{2k-1} c_{0,1,2k-j}(k) (q-1)^j + \mbox{ higher degree terms}.\] 
    
    Clearly,  \begin{align}  \D_k &= \frac{ q^{2k-1}}{(1-q^k)^2}  \D_{k-1} \notag \\
    \implies\F_k &= \frac{ (q-1)^2 q^{2k-1}}{(1-q^k)^2}  \F_{k-1}  \notag \\
   \implies \frac{(1-q^k)^2 }{q^{2k-1}}\F_k &= (q-1)^2 \F_{k-1}. \notag
   \end{align}
   By expanding the initial factor as a Taylor series about $q=1$, we see that
   \begin{multline} \Big( k^2(q-1)^2 - k^3(q-1)^3 + \frac{7k^4+6k^3-k^2}{12} (q-1)^4 \\ + 
   \frac{-3k^5-7k^4-3k^3+k^2}{12} (q-1)^5 + \cdots \Big) \\ \times \Big( \sum_{j=0}^{2k-1} c_{0,1,2k-j}(k) (q-1)^j + \mbox{ higher degree terms} \Big) \\
   = \sum_{r=0}^{2k-3} c_{0,1,2k-2-r}(k-1) (q-1)^{r+2} + \mbox{ higher degree terms} . \label{Rec}
    \end{multline}
    By comparing coefficients of $(q-1)^2$ on either side of~\eqref{Rec}, we find
    \begin{equation} \label{DE2}
      k^2 c_{0,1,2k}(k)  = c_{0,1,2k-2}(k-1). 
    \end{equation}
 Solving the difference equation~\eqref{DE2} with initial condition $c_{0,1,2}(1)=1$ yields
 ~\eqref{c012k}.  
 
   Next, by comparing coefficients of $(q-1)^3$ on either side of~\eqref{Rec}, we find
   \begin{equation} \label{DE3}
      k^2 c_{0,1,2k-1}(k) -k^3 c_{0,1,2k}(k) = c_{0,1,2k-3}(k-1). 
   \end{equation}
 Solving the difference equation~\eqref{DE3} taking into account $c_{0,1,2k}(k) = 1/(k!)^2$ by~\eqref{c012k} and initial condition $c_{0,1,1}(1) = 1$ yields~\eqref{c012km1}.  
   
    Similarly, compare coefficients of $(q-1)^4$ on either side of~\eqref{Rec} to obtain
    \begin{equation} \label{DE4}
    k^2 c_{01,2k-2}(k) - k^3 c_{0,1,2k-1}(k) + \frac{7k^4+6k^3-k^2}{12} c_{0,1,2k}(k) =
    c_{0,1,2k-4}(k-1). \end{equation}
 Solve the difference equation~\eqref{DE4} taking into account both~\eqref{c012k} and~\eqref{c012km1}, with initial condition  $c_{0,1,2}(2) = \frac{11}{16}$ gives~\eqref{c012km2}.
    
 Equation~\eqref{c012km3} follows  
 analogously from the coefficient of $(q-1)^5$ on either side of~\eqref{Rec}.  Obviously, one can continue indefinitely beyond the results listed in this lemma.
  
  \end{proof}
  
  \begin{remark}
    It appears that the coefficients of the powers of $k$ in the $c_{0,1,r}(k)$ form a discernible pattern as well.  Specifically, we have
    \begin{multline*} 
    c_{0,1,2k-j}(k) = \frac{1}{2^j j! (k!)^2} \Big(  k^{2j} + \frac{j^2-10j}{9} k^{2j-1}  \\  + 
    \frac{ \frac{1}{2} j^4 - 12 j^3 - \frac{43}{2} j^2 + 33j }{9^2} k^{2j-2}     + \mbox{ lower degree terms } \Big).
    \end{multline*}
  \end{remark}
  
  \begin{remark}
    Note that the formulas for specific coefficients in the partial fraction decomposition of $\D_k$ are analogous to those of the partial fraction decomposition of $\prod_{j=1}^N (1-q^j)^{-1}$ given by the second author and Zeilberger in~\cite[p. 685, Theorem 3.2]{SZ13}.
  \end{remark}

\begin{theorem} \label{DnkQP}
$D(n,k)$ is a quasipolynomial in $n$ of degree $2k-1$ and quasiperiod $\lcm(1,2,\dots,k)$.
Furthermore, the leading terms of $D(n,k)$ are as follows:
\[ D(n,k) = \frac{1}{(k!)^2 (2k-1)!} n^{2k-1} - \frac{1}{2 (2k-2)! k! (k-2)!} n^{2k-2} + \mbox{ lower degree terms.} \]
\end{theorem}
\begin{proof}
Each term in the series expansion of $\D_k$ of the form $c_{h,j,\ell} q^h / \Phi_j^{\ell}$ yields a series in which the coefficient of $q^n$ is a quasipolynomial in $n$ of quasiperiod $j$.  Thus $D(n,k)$ is a quasipolynomial of quasiperiod $\lcm(1,2,3,\dots,k)$ and of degree $2\lfloor k/j \rfloor - 1$.

The (binomial) series expansion of the term 
\[  \frac{c_{0,1,2k}(k)}{\Phi_1^{2k}} = \frac{1}{(k!)^2 \Phi_1^{2k}} \mbox{ (by Lemma~\ref{L})} \]
in the partial fraction decomposition of $\D_k$
is \begin{equation} \label{LeadBT} \frac{1}{(k!)^2} \sum_{n\geq 0} \binom{n+2k-1}{2k-1} q^n.  \end{equation}
Clearly no other term of the partial fraction decomposition of $\D_k$ yields a coefficient of $q^n$ 
which is of equal or larger degree in $n$; the degree of $\binom{n+2k-1}{2k-1}$ is $2k-1$, and 
its leading coefficient is $1/[(k!)^2(2k-1)!]$.

  The coefficient of $n^{2k-2}$ in $D(n,k)$ arises from two contributions: The second leading term in 
 the expansion of~\eqref{LeadBT} and the leading term in the coefficient of $q^n$ in the series 
 expansion of 
  \[ \frac{c_{0,1,2k-1}(k)}{\Phi_1^{2k-1}} . \]

\end{proof}

\section{Computational considerations}
\subsection{Computation with multisum, Eq.~\eqref{DnkMultisum}} \label{CompMultisum}
We shall derive a result about the computational complexity of
calculating $p(n)$ using~\eqref{PofN} and~\eqref{DnkMultisum} as a result
of the following brief excursion.

Recall the first Rogers--Ramanujan identity~\cite{LJR94}:
\begin{RRSP}  
\begin{equation}
\sum_{k\geq 0}\frac{q^{k^2}}{(1-q)(1-q^2)\cdots(1-q^k)} =
\prod_{j\geq 0} \frac{1}{ (1-q^{5j+1})(1-q^{5j+4} ) },  \label{RRsp}
\end{equation}
\end{RRSP} and its combinatorial interpretation~\cite[p. 35]{M}
\begin{RRC}  For integers $n$, the number of partitions of $n$ into parts which differ from each
other
by at least $2$ equals the number of partitions of $n$ into parts
congruent to $\pm 1 \pmod{5}$.
\end{RRC}

Given the similarity between the general term on the left-hand side of~\eqref{RRsp} and the
generating function for $D(n,k)$, it is not surprising that there are related results.
Noting that 
\begin{equation} \label{RR1term} 
   \frac{q^{k^2}}{(1-q)(1-q^2)\cdots(1-q^k)}=: \sum_{n\geq 0} r_1(n,k) q^n \end{equation}
is the generating function for partitions of length $k$ in which parts differ by at least $2$, if we 
expand~\eqref{RR1term}
as we did in the generating function proof of Lemma~\ref{L1}, the following lemma is immediate:

\begin{lemma} \label{L2}
The number of terms 
in~\eqref{DnkMultisum} equals $r_1(n,k)$.
\end{lemma}

We also note that there is a straightforward bijection
between the partitions of $n-k^2$ in which no part exceeds $k$
(i.e. the partitions indexing the sum in~\eqref{DnkPtn}), and partitions of weight $n$ and length $k$ in
which all parts differ by at least $2$.

Summing over all relevant $k$, that is $1\leq k \leq \lfloor \sqrt{n} \rfloor$, we obtain
\begin{proposition}
The number of terms in~\eqref{PofN}, where $D(n,k)$ is calculated via the multisum~\eqref{DnkMultisum} equals $r_1(n)$, the number of partitions of $n$ into parts which mutually
differ by at least $2$.
\end{proposition}

 Now, on to the complexity result we were leading up to.  Each term on the right hand side of~\eqref{DnkMultisum} involves a product of $k$ factors,
and since $k$ never exceeds $\lfloor \sqrt{n} \rfloor$, the number of integer additions and
multiplications required to calculate $p(n)$ using~\eqref{PofN} via~\eqref{DnkMultisum},
is less than $\lfloor \sqrt{ n }\rfloor r_1(n)$.  Taking into account Lehner's asymptotic 
result~\cite[p.  655, Eq. (12.4) with $a=1$]{L41} (cf.~\cite[p. 97, Ex. 1]{A76}),
\begin{equation}
 r_1(n) \sim \frac{ \sqrt{15+3\sqrt{5}} }{ (60n-1)^{3/4} } \exp\left(  \frac{\pi\sqrt{60n-1}}{15} \right)
\end{equation}
as $n\to\infty$,
we immediately have the following
\begin{theorem}
The number of integer operations (additions and multiplications) required to calculate
$p(n)$ using Equation~\eqref{PofN} via the multisum~\eqref{DnkMultisum} is of
order \[ O(  n^{-1/4}   e^{c \sqrt{n} } ),\] as $n\to\infty$, where $c = 2\pi/\sqrt{15}.$
\end{theorem}

\begin{remark}
The computation of $p(n)$ for all $1\leq n \leq N$ may be done more efficiently using Euler's recurrence~\cite[p. 12, Cor. 1.8]{A76}.  The interested reader may wish to consult~\cite{Cetal} for a 
discussion of efficient algorithms for computing $p(n)$.
\end{remark}

\subsection{Estimation based on quasipolynomial representations of $D(n,k)$}
It is well known that Rademacher's series~\eqref{HRR} converges to $p(n)$ extremely rapidly.  
In fact, just the $k=1$ term provides an excellent approximation to $p(n)$.  How does 
formula~\eqref{PofN} compare to Rademacher's series and how shall we compare them?   
As remarked by Hardy and Ramanujan~\cite[p. 81]{HR18}, $q=1$ is the ``heaviest" singularity 
of the generating function for the partition function on the unit circle, and accordingly the $k=1$ 
term of~\eqref{HRR} contributes more, by far, to the value of $p(n)$ than any other term.  
Analogously, when we consider the formula~\eqref{PofN}, the terms of the partial fraction expansion 
of the $D(n,k)$ with denominators $(q-1)^\ell$ are those that have singularity at $q=1$.  It is these 
terms that give rise to the ``polynomial part" (i.e. the part with quasiperiod $1$) of the expression 
for $D(n,k)$ as a sum of  quasipolynomials.   Accordingly, while we may write
\begin{align*}
 D(n,1) &= n,\\
 D(n,2) &= \frac{(n-1)(2n^2-4n-3)}{48} + (-1)^n \frac{n-1}{16},\\
 D(n,3) &= \frac{(n-3)(6n^4-72n^3+184n^2+192n-235)}{25920} -(-1)^n \frac{n-3}{64} \\
      & \qquad\qquad \notag + \frac{(\omega^n + \omega^{-n})(n-3) + \legendre{n}{3}}{81}\\
      &\mbox{ (where $\omega:=e^{2\pi i/3}$ and $\legendre{n}{3}$ is the Legendre symbol)},\\
 D(n,4) &=   \frac{(n-6)(6n^6-216n^5+2610n^4-10800n^3-2451n^2+60516n-23905)}{17418240}  \\
      &\qquad\qquad \notag + (-1)^n\frac{(n-1)(n-6)(n-11)}{6144}- \frac{(\omega^n + \omega^{-n})(n-6)+ 3\legendre{n}{3}}{243} \\
      &\qquad\qquad\notag + \frac{(i^n + i^{-n})(n-6) + 2(i^{n-1}+i^{1-n})  }{256}, \mbox{ etc.}
\end{align*}
Let us define the ``polynomial approximation"  $\tilde{D}(n,k)$ for each $D(n,k)$ by extracting the terms of quasiperiod 1, so that 
\begin{align*}
 \tilde{D}(n,1) &= n, \\
 \tilde{D}(n,2) &= \frac{(n-1)(2n^2-4n-3)}{48}, \\
 \tilde{D}(n,3) &= \frac{(n-3)(6n^4-72n^3+184n^2+192n-235)}{25920}, \\
 \tilde{D}(n,4) &=   \frac{(n-6)(6n^6-216n^5+2610n^4-10800n^3-2451n^2+60516n-23905)}{17418240},    \\ & \mbox{ etc.,}
\end{align*}
and approximate $p(n)$ by both
\begin{equation} p_D(n):=\sum_{k=1}^{\lfloor \sqrt{n} \rfloor} \tilde{D}(n,k) \end{equation} and
by the $k=1$ term of~\eqref{HRR}:
\begin{equation} p_R(n):=  
\frac{ \cosh \left( \pi 
   \sqrt{\frac 23 \left(n-\frac{1}{24}\right) }\right)}{2\sqrt{3}
   \left(n-\frac{1}{24}\right)}-\frac{\sinh \left( \pi 
   \sqrt{\frac 23 \left( n-\frac{1}{24}\right)}\right)}{2\pi\sqrt{2} \left(n-\frac{1}{24}\right)^{3/2}}.
 \end{equation}
 
 \begin{remark}
 Equivalently, we could define $\tilde{D}(n,k)$ to be minus the residue of $\D_k(q)/q^{n+1}$ at $q=1$,
 or obtain $\tilde{D}(n,k)$ from the principal part of the Laurent series of $\D_k(q)$ about $q=1$.
 \end{remark}
 
Further, numerical experimentation suggests the following rather tight bound, which has been 
verified for $1\leq n \leq 500$: 
\begin{conjecture} For positive integer $n$,
\begin{equation}  \label{pDerror} | p(n) - p_D(n) | \leq  \frac{ 2 (\frac n4)^{\nu - 1} }{\nu \Gamma(\nu + 1) \Gamma( \frac\nu 2 + 1)
\Gamma(\frac\nu 2)}, \end{equation} 
where $\nu = \nu(n) = \frac{27}{50} (2 + \sqrt{n})$.
\end{conjecture} 

We also remark that Rademacher, in discussing error estimates
for truncated versions of his series for $p(n)$, offers~\cite[p. 277, Eq. (121.4)]{HR73}, which
implies
\begin{multline} \label{pRerror} 
  | p(n) - p_R(n) | < \frac{10\sqrt{2}}{99\pi} C^3 + \frac{8\sqrt{2}}{11\pi}n^{-3/2} \left(
  \sinh(C\sqrt{n} ) - C\sqrt{n} \right)  \\+ 2 e^{C\sqrt{n-1}} \frac{4}{n-1} \left( \frac{1}{5\sqrt{3}} + 
  \frac{1}{9\pi\sqrt{2}}(n-1)^{-1/2}  \right),
\end{multline} where $C = \pi \sqrt{2/3}$.
Notice that the right hand side of~\eqref{pDerror} is considerably smaller than that of~\eqref{pRerror}.

Further, Table~\ref{tab} suggests that the absolute errors for $p_D(n)$ and $p_R(n)$ are \emph{very} close to one another, and small compared to $p(n)$, and all by using polynomials with rational coefficients, effectively eliminating the need to resort to transcendental numbers and hyperbolic functions!  (See the graphs in Figures 1 and 2.)

\begin{table}[ht]
\caption{Error in approximation of $p(n)$ by $p_D(n)$ and $p_R(n)$ for selected values of $n$}
\begin{center}
\begin{tabular}{| r | r |  r | r | r | } \hline
$n$ & $p(n)$  & $p_D(n) - p(n)$ & $p_R(n) - p(n)$ & $p_R(n) - p_D(n)$ \\ 
\hline
$5$ & $7$ & $\mathbf{0.25}$ & $0.26210$  & $0.01210 $\\
$10$ & $42$ & $-0.37905$ & $\mathbf{-0.37221}$ & $0.00684  $\\
$15$ & $176$ & $\mathbf{0.39120}$  & $0.56047$ & $ 0.16927 $ \\
$20$ & $627$ & $-1.24394$  & $\mathbf{-1.24232}$ & $ 0.00162 $\\
$25$ & $1958$ & $2.10036$ & $\mathbf{2.09834}$&  $-0.00202  $\\
$30$ & $5604$ & $-3.72589$ & $\mathbf{-3.72044}$ &  $ 0.00545 $\\
$40$ & $37{,}338$ & $-7.39250$ & $\mathbf{-7.39081}$ & $ 0.00170 $\\
$50$ & $204{,}226$ & $\mathbf{-14.9227}$ & $-14.9235$& $-0.00080 $\\
$60$ & $966{,}467$ & $\mathbf{-33.6090}$ & $-33.6385$ & $-0.02946 $\\
$75$ & $8{,}118{,}264$ & $\mathbf{79.2210}$ & $79.2222$ & $ 0.00129  $\\
$100$ & $190{,}569{,}292$ & $-347.2173$ & $\mathbf{-347.2167}$ & $0.00069$ \\
$150$ & $40{,}853{,}235{,}313$ & $-4253.1144$ & $\mathbf{-4253.1138} $ & $0.00058$\\
%$200$ & $3{,}972{,}999{,}029{,}388$ 
$200$ & $> 3.97\times 10^{12}$ & $-36202.1049$ & $\mathbf{-36202.1042} $
 &$ 0.00062$  \\ 
%$300$ & $9{,}253{,}082{,}936{,}723{,}602$ 
$300$ & $> 9.25\times 10^{15}$ & $-1442614.889$ & $\mathbf{-1442614.887}$
& $0.00168$\\
%$500$ & $2{,}300{,}165{,}032{,}574{,}323{,}995{,}027$ 
$500$ & $>2.30 \times 10^{21}$ & $\mathbf{-560997650.0056}$ &
$-560997650.0066$  & $ -0.00093  $\\
\hline
\end{tabular}
\end{center}
\label{tab}
\end{table}

 We further note that 
numerical evidence suggests that local maxima of $p_R(n) - p_D(n)$ occur when $n$ is one less than a perfect square.  
 
 \begin{figure}[h!]\label{graph5}
 \includegraphics[scale=0.4]{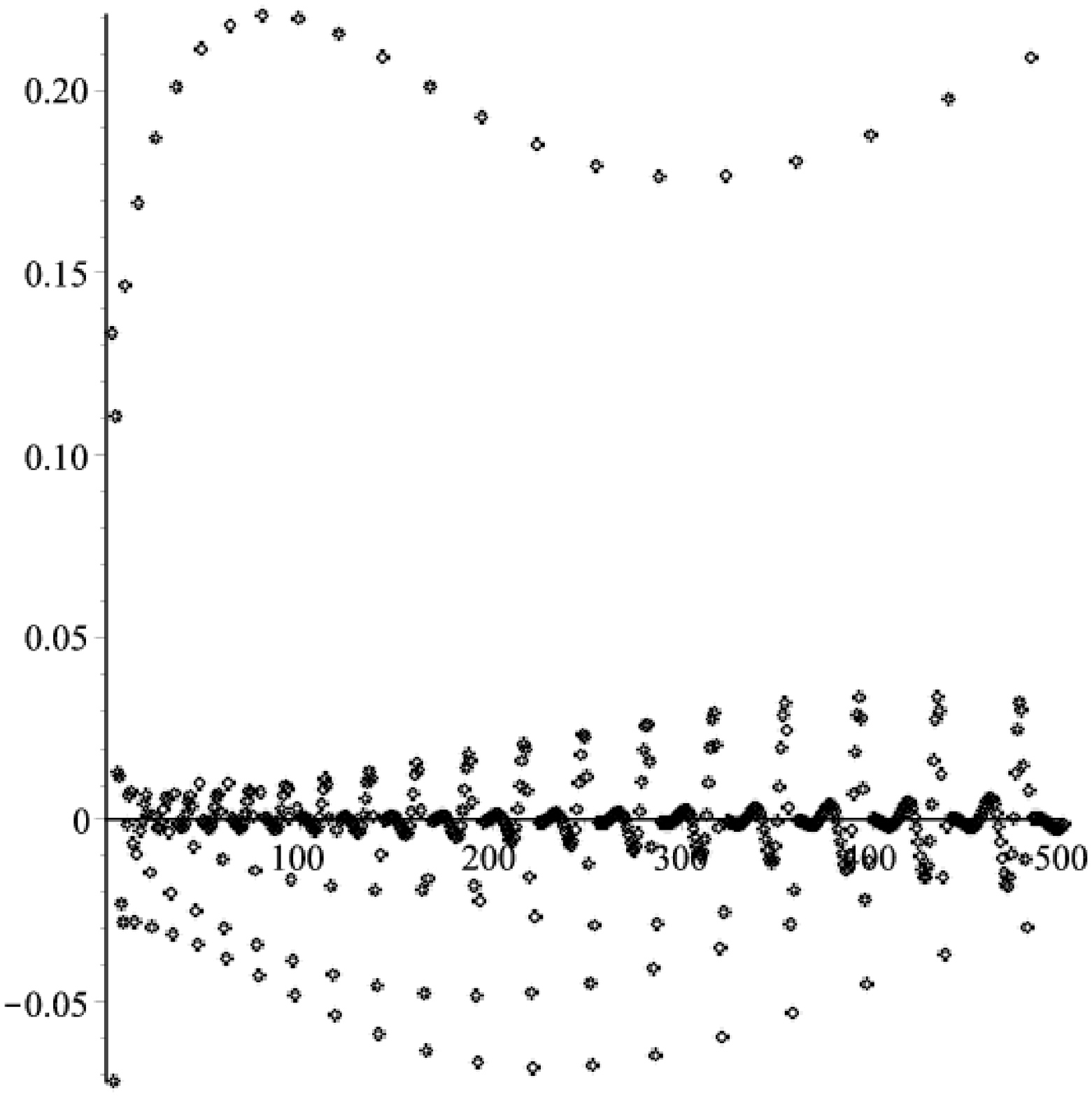}\caption{$p_R(n) - p_D(n)$ for $1\leq n \leq 500$}
\end{figure}

\begin{figure}[h!]\label{graph16}
\includegraphics[scale=0.4]{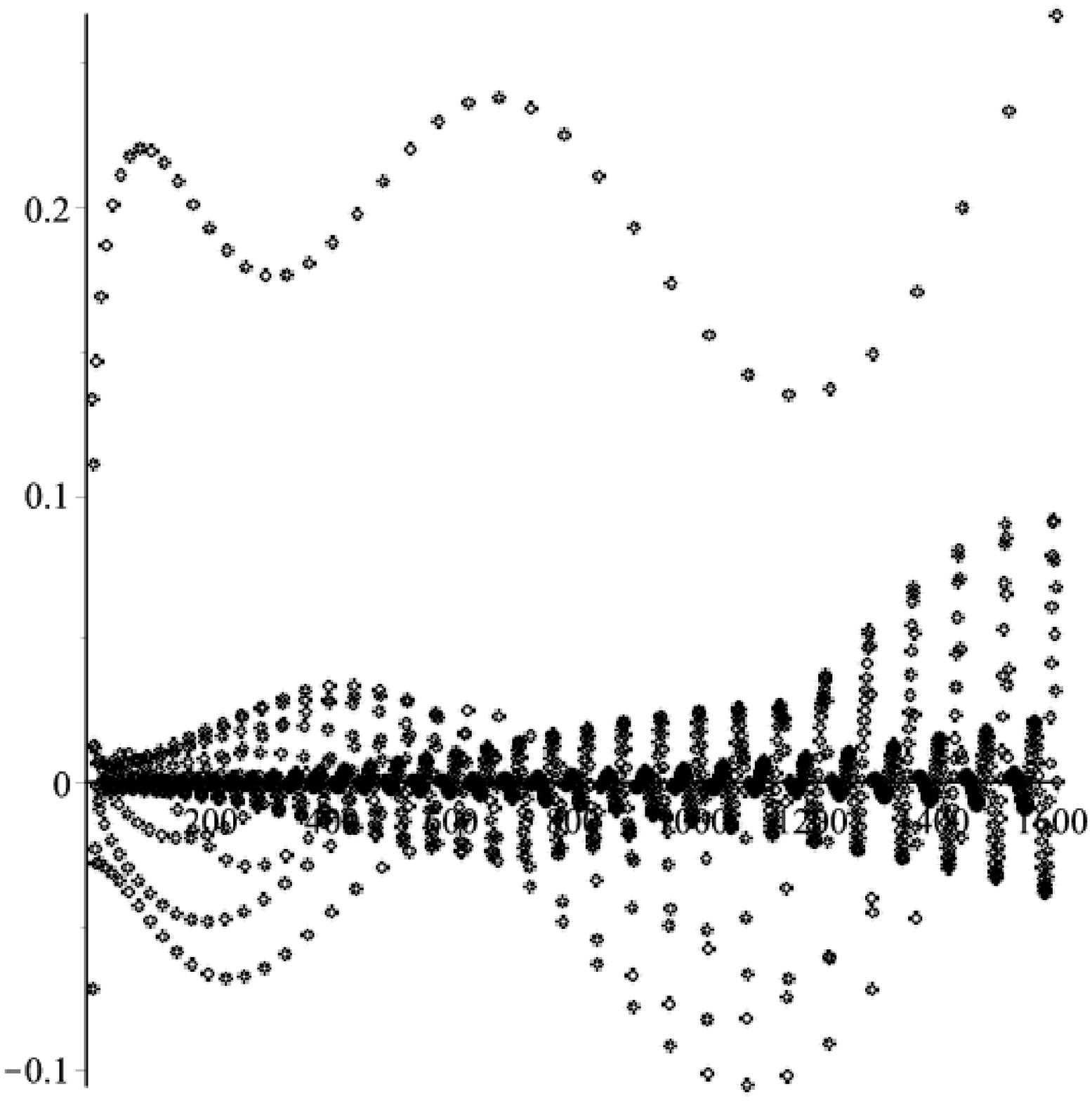}\caption{$p_R(n) - p_D(n)$ for $1\leq n \leq 1600$}
\end{figure}
 
\section{Further related multisum results}
Following the lead suggested by the unexpected appearance of the first Rogers--Ramanujan
identity in Section~\ref{CompMultisum}, we mention several other related results.  Many 
analogous results could be given, but in an effort to keep this article a reasonable length,
we will limit ourselves to three.

Recall the second Rogers--Ramanujan identity~\cite{LJR94}:
\begin{RR2SP}  
\begin{equation} \label{RR2sp}
\sum_{k\geq 0}\frac{q^{k^2+k}}{(1-q)(1-q^2)\cdots(1-q^k)} =
\prod_{j\geq 0} \frac{1}{ (1-q^{5j+2})(1-q^{5j+3} ) },
\end{equation}
\end{RR2SP} and its combinatorial interpretation~\cite[p. 35]{M}.
\begin{RR2C}  For integers $n$, the number of partitions of $n$ into parts which differ from each
other
by at least $2$ and exceed $1$ equals the number of partitions of $n$ into parts
congruent to $\pm 2 \pmod{5}$.
\end{RR2C}
Using reasoning analogous to that of Lemma~\ref{L2}, we obtain the
\begin{proposition}
Let $r_2(n,k)$ denote the number of partitions of $n$ of length $k$ in which all parts exceed 
$1$ and mutually differ by at least $2$.  Then $r_2(n,k)$ equals the number of terms in
\eqref{DnkMultisum} in which $m_k>0$.
\end{proposition}
Of course, summing over $k$ from $1$ to $\lfloor \sqrt{n} \rfloor$ yields the number of
partitions of $n$ enumerated by the second Rogers--Ramanujan identity.

Next we recall perhaps the most famous partition identity of all time:
\begin{EOD}
The number of partitions into distinct parts equals the number of partitions into odd parts.
\end{EOD}

Let $\Delta(n,k)$ denote the number of partitions of $n$ with Durfee square of order $k$ and
all parts distinct.
Then
\begin{theorem}
\begin{equation}
  \Delta(n,k) = \sum_{m_k=0}^{U_k} \sum_{m_{k-1}=1}^{U_{k-1}} 
  \sum_{m_{k-2}=1}^{U_{k-2}} \cdots \sum_{m_2=1}^{U_2}  
 2^{b_1+b_2+\cdots+b_k } \chi(m_1 \neq 0) , \label{MultiSumDistinct}
\end{equation} where
 $U_j$ is defined in~\eqref{Uj}, $\chi(\cdot)$ is the characteristic function,
\begin{equation} \label{m1} m_1 = n - k^2 - \sum_{h=2}^k h m_h, \end{equation}
and
\[  b_i = \begin{cases}  0, & \text{if } m_i = 0\text{ or } 1, \\
      1  ,& \text{if  } m_i > 1\end{cases}.
\] 
\end{theorem}
\begin{proof}
Eq.~\eqref{MultiSumDistinct} was obtained from ~\eqref{DnkMultisum} by eliminating every partition enumerated
by $D(n,k)$ having a repeated part.  To achieve this, the following modifications have been made
to~\eqref{DnkMultisum}:
\begin{enumerate}
   \item Lower bounds for $m_i$, $i=2,3,\dots, k-1$ have been changed from $0$ to $1$, because
$m_i=0$ implies that $\lambda_i = \lambda_{i+1}$.
  \item 
The product $\prod (m_i+1)$ has been replaced by $2^{b_2+\cdots+b_k}$ because $\rho$ must contain at least one part of size $i$ and $\beta$ must contain at most one part of size $i$, where 
the Ferrers graph decomposition
$\lambda \to (\delta, \beta, \rho)$ is as in the combinatorial proof of~Lemma~\ref{L1}.
   \item The factor $(1+n-k^2 - \sum{h m_h})$ has been replaced by $(2^{b_1} \chi(m_1\neq 0))$ to
 account for the necessary restrictions on parts of size $1$.
  \end{enumerate}
\end{proof}

Finally, we note that the similarity between $\mathcal{D}_k$ and 
the general term of another well-known $q$-series,
Ramanujan's third order mock-theta function~\cite{W36}
\begin{equation} \label{MockTheta}
 f(q) = \sum_{k=0}^\infty \frac{q^{k^2}}{(1+q)^2 (1+q^2)^2 \cdots (1+q^k)^2}.
 \end{equation}
Dyson's rank of a partition $\lambda$ is $\lambda_1 - \ell(\lambda)$, the largest part minus
the length~\cite{D44}.
Let $r(n,k)$ denote the coefficient of $q^n$ in the $k$th term of
the right hand side of~\eqref{MockTheta}, the expression 
\[ \frac{q^{k^2}}{(1+q)^2 (1+q^2)^2 \cdots (1+q^k)^2}. \]

The following combinatorial argument shows that $r(n,k)$ is the number of partitions of $n$ with
Durfee square of order $k$ and even rank minus the number of partitions of $n$ with
Durfee square of order $k$ and odd rank:
  Consider an arbitrary partition $\lambda$ of weight $n$.  Decompose 
$\lambda$ into $(\delta,\beta,\rho)$ as in the combinatorial proof of
Lemma~\ref{DnkPtn}.  Note that moving any part between $\beta$ and $\rho$
results in $\lambda'$ having a rank of the same parity as the rank of
$\lambda$.  Therefore, all partitions enumerated by a given product
$\prod(m_i+1)$ have ranks of the same parity.  To establish the parity of 
the rank for a given set $\{ m_1, \dots, m_k \}$, select $\lambda$ such
that the corresponding $\beta$ is the empty partition.  Then the largest
part $\lambda_1$
of $\lambda$ is $\lambda_1 = k + m_1 + m_2 + \cdots + m_k$, and the number
of parts equals $k$.  The result follows, and we have established the 
following
\begin{theorem}
\begin{multline} \label{MockThetaMultisum}
r(n,k) = \sum_{m_k=0}^{U_k} \sum_{m_{k-1}=0}^{U_{k-1}} \cdots \sum_{m_2=0}^{U_2}
 (-1)^{ m_1 + m_2 + m_3 + \cdots + m_k } 
 \prod_{i=1}^k (m_i + 1),
\end{multline}
where $m_1$ is given by~\eqref{m1} and
$U_j$ is given by~\eqref{Uj}.

Thus the coefficient of $q^n$ in $f(q)$~\eqref{MockTheta} is given by
$\sum_{k =1}^{\lfloor \sqrt{n} \rfloor} r(n,k).$
\end{theorem}
\section{Conclusion}
As demonstrated above, while it is tempting to reduce a $(k-1)$-fold multisum to a ``simpler" form, the natural ways to do so brings one to quasipolynomial representations or products of
expressions involving floor functions, which may appear simpler for fixed, numeric $k$ (provided $k$ is sufficiently small) but appear to be more complicated when considering general, symbolic $k$
(or even moderately large numeric $k$).   Furthermore, the quasipolynomial representations obscure the underlying combinatorics, e.g. through the inclusion of 
non-integers and minus signs.  Thus we offer equation~\eqref{DnkMultisum} and its immediate corollary~\eqref{PofN} as truly combinatorial formulas where each term is a positive integer that actually counts partitions.  

  On the other hand, the quasipolynomial representations suggest approximations that are 
simple (in the sense that they are polynomials with rational coefficients), yet striking in their apparent level of accuracy.  We remark that $| p(n) - p_D(n) | < 0.5$ when $n\leq 10$ and when $n= 14, 15$.    The analogous approximation using just the
sum of the quasiperiod $1$ and $2$ parts of $D(n,k)$ approximates $p(n)$ to within $0.5$ of the
true value for all $n\leq 37$ except $n=24$ and $n=36$.   The nearest integer to $\sum_{1\leq k\leq \sqrt{n}} D^{*}(n,k)$, where $D^{*}(n,k)$ is the sum of the parts of $D(n,k)$ of quasiperiod at most $4$, 
suffice to find $p(n)$ exactly for all $n<120$.

\section*{Acknowledgments}
The second author thanks George Andrews and Ken Ono for their interest in this project and for several helpful suggestions, and Marie Jameson for numerous discussions relating to this project.  We also
thank the anonymous referee for very carefully reading our manuscript, and making many helpful suggestions.

\bibliographystyle{amsplain}

\end{document}